\documentclass[leqno]{amsart}
\usepackage{amsmath,amsfonts,amsthm,amssymb,dsfont}
\usepackage[alphabetic]{amsrefs}
\usepackage{hyperref}
\usepackage{times}

\usepackage{enumerate}

\usepackage{mathrsfs}
\usepackage{enumerate, xspace}
\usepackage{graphicx}
\usepackage{color}

\usepackage{pictexwd,dcpic,epsf}

\newtheorem{lem}{Lemma}[section]
\newtheorem{tw}[lem]{Theorem}
\newtheorem*{mtw}{Main Theorem}
\newtheorem{tw2}{Theorem}
\newtheorem{cor}{Corollary}
\newtheorem{prop}[lem]{Proposition}

\theoremstyle{definition}

\newcommand {\xz}{X^{(0)}}
\newcommand {\xj}{X^{(1)}}
\newcommand {\dw}{d_{\mathcal W}}

\newcommand {\ag}{A(\gamma)}

\newcommand{\me}{\medskip}

\theoremstyle{remark}
\newtheorem*{examples}{Examples}

\begin{document}

\title[Small cancellation groups have the Haagerup property]{Infinitely presented small cancellation groups have the~Haagerup property}

\author{Goulnara Arzhantseva}
\address{Universit\"at Wien, Fakult\"at f\"ur Mathematik\\
Oskar-Morgenstern-Platz 1, 1090 Wien, Austria.
}
\address{Erwin Schr\"odinger International Institute for Mathematical Physics\\
Boltzmanngasse~9, 1090 Wien, Austria.
}
\email{goulnara.arzhantseva@univie.ac.at}

\author{Damian Osajda}
\address{Instytut Matematyczny,
Uniwersytet Wroc\l awski (\textup{on leave})\\
pl.\ Grunwaldzki 2/4,
50--384 Wroc{\l}aw, Poland}
\address{Universit\"at Wien, Fakult\"at f\"ur Mathematik\\
Oskar-Morgenstern-Platz 1, 1090 Wien, Austria.
}
\email{dosaj@math.uni.wroc.pl}

\thanks{
G.A.\ was partially supported by the ERC grant ANALYTIC no.\ 259527, and by the Swiss NSF, under Sinergia grant CRSI22-130435.
D.O.\ was partially supported by MNiSW grant N201 541738, and by Narodowe Centrum Nauki, decision no.\ DEC-2012/06/A/ST1/00259.}

\dedicatory{In memory of Kamil Duszenko}

\begin{abstract}
We prove the Haagerup property (= Gromov's a-T-mena\-bi\-li\-ty) for finitely generated groups defined by infinite presentations satisfying the $C'(1/6)$--small cancellation condition.
We deduce that these groups are coarsely embeddable into a Hilbert space and that the strong Baum-Connes conjecture holds for them. The result is a first non-trivial advancement in understanding groups with such properties among infinitely presented non-amenable direct limits of hyperbolic groups.
The proof uses the structure of a space with walls introduced by Wise.
As the main step we show that $C'(1/6)$--complexes satisfy the linear separation property.
\end{abstract}

\maketitle

\section{Introduction}
\label{s:intro}

A second countable, locally compact group $G$ has the \emph{Haagerup property}  (or $G$ is \emph{a-T-menable}  in the sense of Gromov) if it possesses
a proper continuous affine isometric action on a Hilbert space. The concept first appeared in the seminal paper of Haagerup \cite{Haa}, where
this  property was proved for finitely generated free groups. Regarded as a weakening of von Neumann's amenability and a strong negation of Kazhdan's property (T), the Haagerup property has been  revealed independently in harmonic analysis, non-commutative geometry, and ergodic theory \cite{AW,Cho,BoJaS,BR}, \cite[4.5.C]{Gro88}, \cite[7.A and 7.E]{Gro93}.
A major breakthrough was a spectacular result of Higson and Kasparov \cite{HK} establishing the strong Baum-Connes conjecture  (and, hence, the Baum-Connes conjecture with coefficients)
for groups with the Haagerup property. It follows that the Novikov higher signature conjecture and, for discrete torsion-free groups, the Kadison-Kaplansky idempotents conjecture hold for
these groups. Nowadays, many groups have been shown to have the Haagerup property and  several significant applications in K-theory and topology have been discovered \cite{ChCJJV,MislinValette}, making groups with the Haagerup property increasingly fundamental to study.

Finitely presented groups defined by a presentation with the classical small cancellation condition $C'(\lambda)$ for $\lambda \leqslant 1/6$ (see \cite{LS} for the definition)
satisfy the Haagerup property by a  result of Wise \cite{W-sc}. 

 The appearance  of \emph{infinitely presented} small cancellation groups can be traced back to nu\-merous embedding results (the idea is attributed to Britton~\cite[p.172]{MillerSchupp}):
the small cancellation condition over free products was systematically used to get an embedding of a countable group into a finitely generated group with required properties \cite[Ch.V]{LS}.
A more recent example is the Thomas-Velic\-kovic construction of a group with two distinct asymptotic cones \cite{TV}.
However, the general theory of infinitely presented small cancellation groups is much less developed than the one for the finitely presented counterpart
(see e.g.\ \cites{LS,W-sc,W-qch} for results and further references). This is related to the fact that such infinitely presented groups form a kind of a borderline for many geometric or analytic properties.
For instance, Gromov's monster groups\footnote{These are finitely generated groups which contain an expander family of graphs in their Cayley graphs.} \cite{Gro,AD}
do not satisfy the Baum-Connes conjecture with coefficients~\cite{HLS} while they are direct limits of finitely presented graphical small cancellation groups. The latter groups may not have the Haagerup property (see e.g.\ \cite[Proposition 3.1]{OllivierWise}), but they are Gromov hyperbolic.
Therefore, they satisfy the Baum-Connes conjecture with coefficients by a recent deep work of Lafforgue \cite{Lafforgue:hyp}.
Also, Gromov's monster groups admit no coarse embeddings into a Hilbert space and, hence, are not coarsely amenable \cite{Gro, AD}. Again, Gromov hyperbolic groups are known to possess both properties \cite{Yu}.

Even for the simplest case of classical small cancellation  infinitely presented groups (as considered in \cite{TV}) the questions about various Baum-Connes conjectures (see \cite{Valette:intro} for diverse variants of the conjecture) and the coarse embeddability into a Hilbert space have remained open. 

The coarse embeddability into a Hilbert space is implied by the finiteness of the asymptotic dimension~\cite{Yu}. Although
we do expect that classical small cancellation infinitely presented groups have finite asymptotic dimension\footnote{See \cite[Problem 3.16]{Dra}, \cite[Problem 6.1]{Osin}, \cite[Question 4.2]{Dusz} for specifications of the question.},  this cannot be used to obtain deepest possible analytic results such as the strong Baum-Connes conjecture~\cite{MN} (which is strictly stronger than the Baum-Connes conjecture with coefficients). Indeed, a discrete subgroup of finite covolume in $Sp(n, 1)$ is a group with finite asymptotic dimension
which does not satisfy the strong Baum-Connes conjecture~\cite{Sk}.

In this paper, we answer the questions concerning Baum-Connes conjectures and the coarse embeddability into a Hilbert space by proving the following stronger result.

\me

\begin{mtw}
\label{t:main}
  Finitely generated groups defined by infinite $C'(1/6)$--small cancellation presentations have the Haagerup property.
\end{mtw}
As an immediate consequence we obtain the following.
\begin{cor}
  Finitely generated groups defined by infinite $C'(1/6)$--small cancellation presentations are coarsely embeddable into a Hilbert space.
\end{cor}

Moreover, using results of \cite{HiKa}, we have:
\begin{cor}
  The strong Baum-Connes conjecture holds for finitely generated groups defined by infinite $C'(1/6)$--small cancellation presentations.
\end{cor}

\medskip

Our approach to proving Main Theorem is to show a stronger result:
A group acting properly on a simply connected $C'(1/6)$--complex, acts properly on a space with walls. 
The concept of a space with walls was introduced by Haglund-Paulin \cite{HP} (cf.\ Section~\ref{s:walls}).
It is an observation by
Bo\. zejko-Januszkie\-wicz-Spatzier \cite{BoJaS} (implicitly, without the notion of  a ``space with walls" yet) and later, independently, by Haglund-Paulin-Valette (unpublished --- compare \cite[Introduction]{ChMV}), that a finitely generated group admitting a proper action on a space with walls has the Haagerup property.
We define walls on the $0$--skeleton of the corresponding $C'(1/6)$--complex, using the construction of Wise \cite{W-sc} (cf.\ Section~\ref{s:walls}). The main difficulty is to show the properness. To do this we prove the following general result about complexes (see Theorem~\ref{p:lsp} for the precise statement). Recall that the \emph{linear separation property} says that the wall pseudo-metric and the path metric, both considered on the $0$--skeleton of the complex, are bi-Lipschitz 
equivalent.
\begin{tw2}
\label{t:lsp}
  Simply connected $C'(1/6)$--complexes satisfy the linear separation property.
\end{tw2}

This result is of independent interest. Note that for complexes satisfying the $B(6)$--condition, introduced and extensively explored by Wise \cites{W-sc,W-qch}, the linear separation property does not hold in general. Moreover,  such complexes might not admit any proper separation property --- see Section~\ref{s:final}  
(where we also explain that our results are not immediate consequences of Wise's work).

From Theorem~\ref{t:lsp} it follows that groups acting properly on simply connected $C'(1/6)$--complexes act properly on spaces with walls
 (Theorem~\ref{t:haag}). This implies immediately Main Theorem. This also extends a result of Wise \cite[Theorem 14.2]{W-sc} on non-satis\-fiability of Kazhdan's property (T) for such infinite groups --- see Corollary~\ref{c:noT} in Section~\ref{s:Haa}.
In addition, the linear separation property yields results on the affine isometric group actions on
$L^p$ spaces --- see Corollary~\ref{c:orbit}.
\medskip

Our main result holds as well for certain groups with more general graphical small cancellation presentations. Note however that some graphical small cancellation groups 
satisfy Kazhdan's property (T), and thus do not have the Haagerup property --- cf.\ e.g.\ \cite{Gro,OllivierWise}.  Therefore, besides providing new results, the current paper plays also the role of an initial step in a wider program for
distinguishing groups with the Haagerup property among infinitely presented non-amenable direct limits of hyperbolic groups.
\medskip

\noindent
{\bf Acknowledgment.} We thank Dominik Gruber, Yves Cornulier  and Alain Valette for valuable remarks improving the manuscript.
We thank the anonymous referee for useful comments and for pointing out to us Corollary~\ref{c:orbit}.

\section{Preliminaries}
\label{s:prelim}

A standard reference for the \emph{classical small cancellation} theory considered in this paper is the book \cite{LS}.
In what follows however we will mostly deal with an equivalent approach, focusing on CW complexes, following the notations from \cites{W-sc,W-qch}.
\medskip

All complexes in this paper are simply connected \emph{combinatorial} $2$--dimensional CW complexes, i.e.\ restrictions of attaching maps to open edges are homeomorphisms onto open cells.
We assume that if in such a complex $X$ two cells are attached along a common boundary then they are equal, i.e.\ e.g.\ that $2$--cells are determined uniquely by images of their attaching maps. 
Thus, we do not distinguish usually between a $2$--cell and its boundary, being a cycle --- we denote both by $r$ and call them \emph{relators}.
Note that an attaching map $r\to X$ need not to be injective.  However, the injectivity indeed holds in the case we consider below.
Moreover, we assume that relators have even length. This is not a major restriction since one can always pass to a complex whose edges are subdivided in two.
Throughout the article, if not specified otherwise, we consider the \emph{path metric}, denoted $d(\cdot,\cdot)$, defined on the $0$--skeleton $X^{(0)}$ of $X$ by (combinatorial) paths in $X^{(1)}$. \emph{Geodesics} are the shortest paths in $X^{(1)}$ for this metric.
By a (generalized) \emph{path} we mean a cellular map $p\to X$, from a subdivision of an interval to $X$.

\medskip
A path $p \to X$ is a \emph{piece} if there are $2$--cells $r,r'$ such that $p\to X$ factors as $p \to r \to X$ and as $p\to r' \to X$, but there is no isomorphism $r \to r'$ that makes the following diagram commutative.
   $$  \begindc{\commdiag}[2]
\obj(12,1)[a]{$r$}
\obj(35,1)[b]{$X$}
\obj(35,17)[c]{$r'$}
\obj(12,16)[d]{$p$}
\obj(12,17)[d']{}
\mor{a}{b}{}
\mor{c}{b}{}
\mor{a}{c}{}
\mor{d'}{c}{}
\mor{d'}{a}{}
\enddc  $$

\noindent
 This means that $p$ occurs in $r$ and $r'$ in two essentially distinct ways.

Let $\lambda \in (0,1)$.
We say that the complex $X$ satisfies the \emph{$C'(\lambda)$--small cancellation condition} (or, shortly, the \emph{$C'(\lambda)$--condition}; or we say that $X$ is a \emph{$C'(\lambda)$--complex}) if every piece $p\to X$ factorizing through $p\to r \to X$ has \emph{length} $|p|<\lambda |r|$ (where $|s|$ is the number of edges in the path $s$). We say that $X$ satisfies the \emph{$B(6)$--condition} if every path factorizing through $r$ and being a concatenation of at most $3$ pieces has length at most $|r|/2$. Note (cf.\ \cite[Section 2.1]{W-sc}) that the $C'(1/6)$--condition implies the $B(6)$--condition. Under the $B(6)$--condition, the attaching maps $r\to X$ for relators are injective --- see e.g.\ \cite[Corollary 2.9]{W-sc}. Thanks to this, we may view relators as embedded cycles in the $1$--skeleton $X^{(1)}$ of $X$.
\medskip

In this paper, we work with a group $G$ defined by an infinite presentation
\begin{align}
 \label{e:p1}
  G=\langle S \; | \; r_1,r_2,r_3,\ldots \rangle,
\end{align}
with a finite symmetric generating set $S$ and (freely) cyclically reduced \emph{relators} $r_i$. There is a combinatorial $2$--dimensional CW complex, the \emph{Cayley complex} $X$, associated 
with the presentation (\ref{e:p1}). It is
defined as follows.
The $1$--skeleton $X^{(1)}$ of
$X$ is the Cayley graph (with respect to the generating set $S$) of $G$.
The $2$--cells of $X$ have boundary cycles labeled by relators $r_i$ and are attached to the $1$--skeleton by maps preserving labeling (of the Cayley graph).
We say that the presentation (\ref{e:p1}) is a \emph{$C'(\lambda)$--small cancellation presentation} when the corresponding Cayley complex $X$ satisfies the $C'(\lambda)$--condition.

\begin{examples} Here are a few concrete examples of infinite small cancellation presentations defining groups with various
unusual properties.

\begin{itemize}
\item[(i)](Pride)  For each positive integer $n,$ let $u_n, v_n$ be words in $a^n$ and $b^n,$ and let
$$
G=\langle a,b \mid au_1, bv_1, au_2, bv_2, au_3, bv_3, \ldots \rangle.
$$ 
An appropriate choice of $u_n, v_n$  gives an infinite $C'(1/6)$--small cancellation presentation of a non-trivial group $G.$
For instance, one can take
$u_n=(a^nb^n)^{10}, v_n=(a^nb^{2n})^{10}$ for $n\geqslant 1.$
By construction, $G$ has no proper subgroups of finite index. Indeed, such a subgroup has to contain a normal closure of $a^n$ and $b^n$ for some $n$, which coincides with $G$
due to the chosen relators.  In particular, $G$ is not residually finite~\cite{Pride}. 
Every finite $C'(1/6)$--small cancellation presentation defines a residually finite group~\cite{W-qch}. 
Whether or not there exists a non residually finite Gromov hyperbolic group is a major open question in geometric group theory.\\[-2pt]

\item[(ii)](Thomas--Velic\-kovic) The following infinite presentations satisfy \linebreak the $C'(1/6)$--small cancellation condition:
$$
G_{I,k}=\langle a,b \mid (a^nb^n)^k, n\in I \rangle,   
$$
where $I\subseteq \mathbb N$ is a given infinite subset and $k\geqslant 7$ is a fixed integer.
These are first examples of finitely generated groups with two distinct asymptotic cones~\cite{TV} (which arise with respect to two appropriately chosen, depending on $I,$ distinct 
ultrafilters on $\mathbb N$). We refer the reader to \cite{DrutuSapir,ErschlerOsin} for more results in this direction using infinitely presented small cancellation groups in a crucial way.\\[-2pt]
 
\item[(iii)](Rips) Given a finitely generated group $Q=\langle a_1,\ldots, a_m \mid r_1, r_2, r_3, \ldots \rangle$ there is
a $C'(1/6)$-small cancellation group $G$ and a 2-generated subgroup $N$ of $G$ so that $G/N\cong Q.$ Indeed,
let $G$ be given by generators $a_1, \ldots, a_m, x, y$ and relators
$$
A_j=(xy)^{80j+1}xy^2(xy)^{80j+2}xy^2\ldots (xy)^{80(j+1)}xy^2,  j=1,2,\ldots, \\
$$
where $A_1, A_2, \ldots$ is the sequence of words: $$a_i^{\pm 1}xa_i^{\mp 1} (i=1,\ldots m), a_i^{\pm 1}ya_i^{\mp 1} (i=1,\ldots m), r_1, r_2, r_3,  \ldots .$$
This presentation of $G$ satisfies the $C'(1/6)$--small cancellation condition, the required $N$ is the subgroup of $G$ generated by $x, y,$ and $G$ is finitely presented if and only if $Q$ is.
A specific choice of $Q$ produces (finitely presented) small cancellation groups with exotic algebraic and algorithmic properties of certain subgroups~\cite{Rips}, see also~\cite{BaumslagMillerShort}, and~\cite{W-rf,BelegradekOsin,OllivierWise} for variants of Rips construction (which provide $G$ and/or $N$ with additional properties). 
\end{itemize}
\end{examples}

Observe that a given infinite small cancellation presentation $G=\langle a,b \; | \; r_1,r_2,r_3,\ldots \rangle$
yields many distinct  small cancellation groups:
For $I\subseteq \mathbb N$ define $G_I=\langle a,b \; | \; r_i, i \in I \rangle.$ Then the family
$
\left\{ G_I\right\}_{I\subseteq \mathbb N} 
$
contains continuum many non-isomorphic small cancellation groups (use the small cancellation condition to show that there is no
group isomorphism mapping $a\mapsto a, b\mapsto b$ and the cardinality argument to conclude as there are at most countably many other possible generators). 

\subsection{Local-to-global density principle}
\label{s:prel}

Here we provide a simple trick that will allow us to deal with different sizes of relators in Section~\ref{s:properness}.

Let $\gamma$ be a simple path in $X^{(1)}$.
For a subcomplex $B$ of $\gamma$, by $E(B)$ we denote the set of edges of $B$.
Let $\mathcal U$ be a family of nontrivial subpaths of $\gamma$, and let $A$ be a subcomplex of $\bigcup \mathcal U$ (that is, of the union $\bigcup_{U\in \mathcal U}U$).

\begin{lem}[Local-to-global density principle]
\label{l:lgd}
Assume that there exists $C\geqslant 0$, such that
\begin{align*}
  \frac{|E(A)\cap E(U)|}{|E(U)|}\geqslant C,
\end{align*}
for every $U\in \mathcal U$.
Then $|E(A)|\geqslant (C/2)|E(\bigcup \mathcal U)|$.
\end{lem}

\begin{proof}
  Let $\mathcal U'\subseteq \mathcal U$ be a minimal cover of $\bigcup \mathcal U$. Then there are two subfamilies $\mathcal U_1', \mathcal U_2'$ of $\mathcal U'$
  with the following properties:
  \begin{enumerate}
    \item $\mathcal U_i'$ consists of pairwise disjoint paths, $i=1,2$;
    \item $\mathcal U_1' \cup \mathcal U_2'=\mathcal U'$.
  \end{enumerate}
  Without loss of generality we may assume that $|E(\bigcup \mathcal U_1')|\geqslant |E(\bigcup \mathcal U')|/2$.
  Then
  \begin{align*}
  |E(A)|&\geqslant |E(A)\cap E(\bigcup \mathcal U_1')|=\sum_{U\in \mathcal U_1'} |E(A)\cap E(U)| \geqslant\\
   &\geqslant \sum_{U\in \mathcal U_1'} C|E(U)|
  = C|E(\bigcup \mathcal U_1')|\geqslant C|E(\bigcup \mathcal U')|/2=\frac{C}{2}|E(\bigcup \mathcal U)|.
  \end{align*}
\end{proof}

\section{Walls}
\label{s:walls}

Let $X$ be a complex satisfying the $B(6)$--condition.
In this section, we equip the $0$--skeleton $X^{(0)}$ of $X$ with the structure of a space with walls $(X^{(0)}, \mathcal{W})$. We use the walls defined by Wise  \cite{W-sc}.
\medskip

\noindent
{\bf Remark.} Many considerations in \cite{W-sc} concern finitely presented groups. Nevertheless, all the results about walls stated below are provided there under no further assumptions on the complex $X$. \medskip

Let $\mathcal W$ be a family of partitions (called \emph{walls}) of a set $Y$ into two classes.
The pair $(Y,\mathcal W)$ is called a \emph{space with walls} (cf.\ e.g.\ \cite{ChMV}) if the following holds.
For every two distinct points $x,y\in Y$, the number of walls separating $x$ from $y$ (called the \emph{wall pseudo-metric}), denoted by $\dw(x,y)$, is finite.
\medskip

Now we define walls for $\xz$. For a tentative abuse of notation we denote by ``walls" some sets of edges of $\xj$. Then we show that they indeed define walls.
Following Wise \cite{W-sc}, we say that two edges are related if they are opposite in some $2$--cell. The equivalence class of the transitive closure of such relation is called a \emph{wall}.

\begin{lem}[{\cite[Lemma 3.13]{W-sc}}]
  \label{l:wallsep}
  Removing all open edges from a given wall disconnects $\xj$ into exactly two components.
\end{lem}

Thus, we define the family $\mathcal W$ for $\xz$ as the partitions of $\xz$ into sets of vertices in the components described by the lemma above.

\begin{prop}
  \label{p:sww}
  With the system of walls defined as above, $(\xz,\mathcal W)$ becomes a space with walls.
\end{prop}
\begin{proof}
  Since, for any two vertices, there exists a path in $\xj$ connecting them, we get that the number of walls separating those two vertices is finite.
\end{proof}
We recall two further results on walls that will be used in Section~\ref{s:properness}.
The \emph{hypercarrier} of a wall $w$ is the $1$--skeleton of the subcomplex of $X$ consisting of all closed $2$--cells containing edges in $w$ or of a single edge $e$ if $w=\{ e \}$.

\begin{tw}[{\cite[Theorem 3.18]{W-sc}}]
  \label{l:carrconv}
  Each hypercarrier is a convex subcomplex of $\xj$, that is, any geodesic connecting vertices of a hypercarrier is contained in this hypercarrier.
\end{tw}

For a wall $w$, its \emph{hypergraph} $\Gamma_w$ is a graph defined as follows. Vertices of $\Gamma_w$ are edges in $w$, and edges correspond to $2$--cells containing opposite edges in $w$.

\begin{lem}[{\cite[Corollary 3.12]{W-sc}}]
\label{l:hypergraph}
Each hypergraph is a tree.
\end{lem}

\section{Linear separation property}
\label{s:properness}
\noindent
\emph{From now on, unless stated otherwise, each complex $X$ considered in this paper, satisfies the $C'(\lambda)$--condition, for some $\lambda \in (0,\frac{1}{6}]$, and its $0$--skeleton is equipped with the structure of a space with walls $(X^{(0)},\mathcal W)$ described in Section~\ref{s:walls}.}
\medskip

In this section, we show that complexes satisfying $C'(1/6)$--condition satisfy the \emph{linear separation property} (Theorem~\ref{t:lsp} in Introduction, and Theorem~\ref{p:lsp} below) stating that the wall pseudo-metric on $\xz$ is bi-Lipschitz equivalent to the path metric (cf.\ e.g.\ \cite[Section 5.11]{W-qch}). Note that the linear separation property does not hold in general for $B(6)$--complexes --- see Section~\ref{s:final}.
\medskip

Let $p,q$ be two distinct vertices in $X$.
It is clear that
\begin{align*}
  \dw(p,q) \leqslant d(p,q).
\end{align*}
For the rest of this section our aim is to prove an opposite (up to a scaling constant) inequality.
\medskip

Let $\gamma$ be a geodesic in $X$ (that is, in its $1$--skeleton $\xj$) with endpoints $p,q$. Let $A(\gamma)$ denote the set of edges in $\gamma$ whose walls meet $\gamma$ in only one edge (in particular such walls separate $p$ from $q$). Clearly $\dw(p,q)\geqslant |\ag|$.
We thus estimate $\dw(p,q)$ by closely studying the set $\ag$. The estimate is first provided locally, and then we use the local-to-global density principle (Lemma~\ref{l:lgd}) to obtain a global bound.

\subsection{Local estimate on $|\ag|$.}
\label{s:local}
For a local estimate we need to define neighborhoods $N_e$ --- \emph{relator neighborhoods in $\gamma$} --- one for every edge $e$ in $\gamma$, for which the number $|E(N_e)\cap \ag|$ can be bounded from below.

For a given edge $e$ of $\gamma$ we define a corresponding relator neighborhood $N_e$ as follows.
If $e\in \ag$ then $N_e=\{ e \}$. Otherwise, we proceed in the following way.
\medskip

Since $e$ is not in $\ag$, its wall $w$ crosses $\gamma$ in at least one more edge. In the wall $w$, choose an edge $e'\subseteq \gamma$ being closest to $e\neq e'$. The hypergraph $\Gamma_w$ of the wall $w$ is a tree by Lemma~\ref{l:hypergraph}. Consider the geodesic between vertices $e$ and $e'$ in $\Gamma_w$. Let $r$ be the relator corresponding to an edge in $\Gamma_w$ lying on this geodesic and containing $e$. 
Two edges in a wall contained in a single relator (that is, opposite in that relator) do not lie on a geodesic in $\xj$ (by Theorem~\ref{l:carrconv}).
Since $\gamma$ is a geodesic, we have that $e'$ is not in $r$. Thus, let $e''$ be a vertex (edge in $X$) on the geodesic in $\Gamma_w$ contained in $r$ (considered as an edge in $\Gamma_w$). Consequently, let $r'$ be the other relator containing $e''$ and corresponding to an edge in the geodesic in $\Gamma_w$.

We define $N_e$ as the intersection $r\cap \gamma$, that is, as the maximal subpath of $\gamma$ contained in the relator $r$. Observe that the choice of $N_e$ is not unique.
In the rest of this section we estimate the number of edges in $N_e$ belonging to $\ag$.
\medskip

Denote by $p',q'$ the endpoints of $N_e$, such that $p'$ is closer to $p$. We begin with an auxiliary lemma.

\begin{figure}[h!]
\centering
\includegraphics[width=0.9\textwidth]{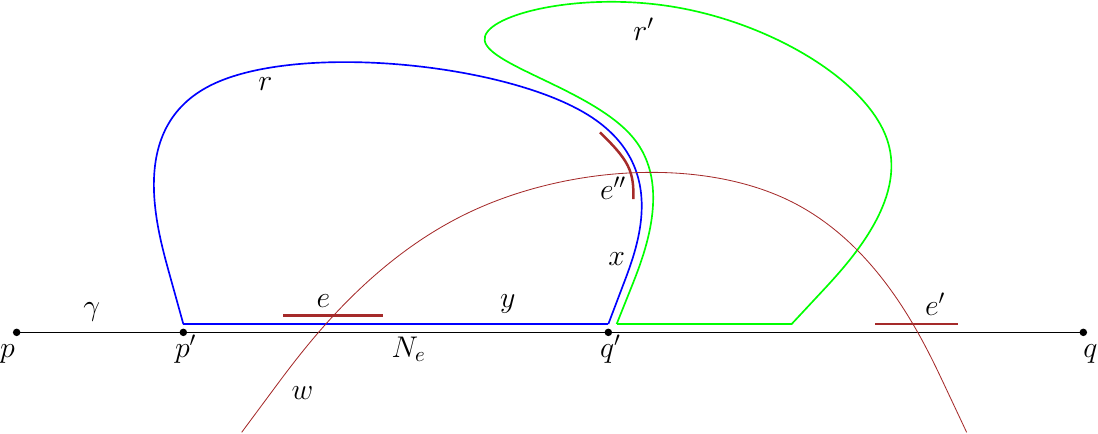}
\caption{The situation in Lemma~\ref{l:prop1}.}
\label{fig_1}
\end{figure}

\begin{lem}
 \label{l:prop1}
  Assume that $q'$ lies (on $\gamma$) between $e$ and $e'$.
  Then we have:
  \begin{align}
    \label{e:l1a}
    d(p',q')>d(e,q')>\left( \frac{1}{2}-\lambda \right) |r|,
  \end{align}
  \begin{align}
    \label{e:l1b}
    d(e,p')<2\lambda d(p',q')-1.
  \end{align}
\end{lem}
\begin{proof}
  Let $x=d(e'',q')$ and $y=d(e,q')$ --- see Figure~\ref{fig_1}. By definition of the wall $w$, we have
  \begin{align}
    \label{e:l1-1}
    y+x+1=\frac{|r|}{2}.
  \end{align}
  Since $N_e$ is a geodesic we obtain
  \begin{align}
    \label{e:l1-2}
    d(p',q')\leqslant\frac{|r|}{2}.
  \end{align}

  The relator $r'$ belongs to a hypercarrier of $w$, and $q',e'$ are endpoints of a geodesic lying both in the hypercarrier.
  By the convexity (Theorem~\ref{l:carrconv}) and by the tree-like structure of hypercarriers (Lemma~\ref{l:hypergraph}) we obtain that $q'\in r'$.
  Thus, the path in $r$ joining $q'$ and $e''$, including $e''$, is contained in $r'$. It follows that this path, of length $x+1$, is a piece and hence, by the $C'(\lambda)$--small cancellation condition, we have
  \begin{align}
    \label{e:l1-3}
    x+1<\lambda |r|.
  \end{align}
  Combining (\ref{e:l1-1}) and (\ref{e:l1-3}) we obtain
  \begin{align*}
    d(p',q')>y=\frac{|r|}{2}-(x+1)>\frac{|r|}{2}-\lambda |r|,
  \end{align*}
  that proves (\ref{e:l1a}).
  Combining this with (\ref{e:l1-2}) we obtain
  \begin{align*}
    \frac{y}{d(p',q')}> \frac{|r|/2-\lambda |r|}{|r|/2}=1-2\lambda.
  \end{align*}
  Thus,
  \begin{align*}
    d(e,p')=d(p',q')-y-1<2\lambda d(p',q')-1,
  \end{align*}
  that finishes the proof.
\end{proof}

\begin{lem}[Local density of $\ag$]
  \label{l:prop2}
The number of edges in $N_e$, whose walls separate $p$ from $q$ is estimated as follows:
\begin{align*}
  |E(N_e)\cap \ag|\geqslant \frac{1-6\lambda+4\lambda^2}{1-2\lambda}\cdot|E(N_e)|.
\end{align*}
\end{lem}
\begin{proof}
  If $e\in \ag$, then $N_e=\{ e\}$ and the lemma is trivially true. Thus, for the rest of the proof we assume that this is not the case and we use the notations introduced above, that is: $e'',p',q',r,r'$.
  To estimate the number of edges in $N_e$ that belong to $\ag$, that is, $|E(N_e)\cap \ag|$ we explore the set of edges $f$ in $N_e$ not belonging to $\ag$.
  \medskip

  For such an $f$, let $f'\subseteq \gamma$ be a closest edge in the same wall $w_f$ as $f$. Again, there is a relator $r_f$ containing $f$, whose corresponding edge in the hypergraph $\Gamma_{w_f}$ lies on the geodesic between $f$ and $f'$. Let $p''$ and $q''$ denote the endpoints of the subpath $r_f\cap \gamma$, with $p''$ closer to $p$.
  There are two cases for such an $r_f$, that we treat separately.
  \medskip

  \noindent
  \emph{Case ``Up":} In this case, we have $r_f=r$.
  Then, by Lemma~\ref{l:prop1}(\ref{e:l1b}), we have
  \begin{align}
    \label{e:l2-1}
    d(f,q')< 2\lambda d(p',q')-1,
  \end{align}
  or
  \begin{align}
    \label{e:l2-2}
    d(f,p')< 2\lambda d(p',q')-1.
  \end{align}

  \medskip

  \noindent
  \emph{Case ``Down":} In this case, we have that $r_f\neq r$. Without loss of generality we may assume that $q''$ lies (on $\gamma$) between $f$ and $f'$ --- see Figures~\ref{fig_2} \&~\ref{fig_3}.

  \begin{figure}[h!]
\centering
\includegraphics[width=0.9\textwidth]{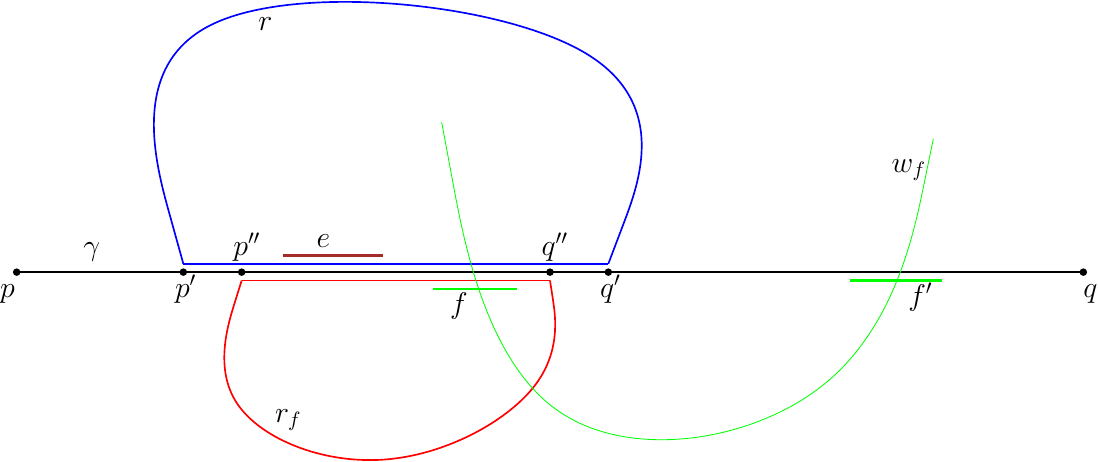}
\caption{The impossible case ``Down".}
\label{fig_2}
\end{figure}

  First, suppose that
  $q''\in N_e$ --- see Figure~\ref{fig_2}.
  Then the subpath of $\gamma$ between $f$ and $q''$, including $f$, is a piece. Thus, by the $C'(\lambda)$--small cancellation condition, we have that $d(f,q'')<\lambda |r_f|$.
  However, by Lemma~\ref{l:prop1}(\ref{e:l1a}) we have that $d(f,q'')>(1/2-\lambda)|r_f|$, leading to a contradiction for $\lambda \leqslant1/4$.

  \begin{figure}[h!]
  \centering
  \includegraphics[width=0.9\textwidth]{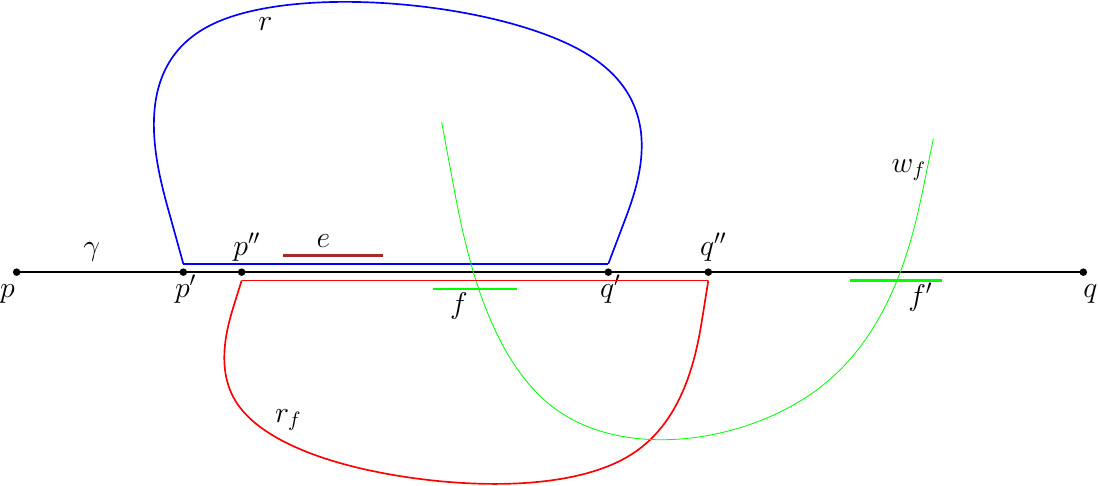}
  \caption{The possible case ``Down".}
  \label{fig_3}
  \end{figure}

  Thus, $q''$ lies (on $\gamma$) between $q'$ and $q$ --- see Figure~\ref{fig_3}.
  It follows that the subpath of $\gamma$ between $f$ and $q'$, including $f$, is a piece. By the $C'(\lambda)$--small cancellation condition we have
  \begin{align*}
     d(f,q')+1 <\lambda |r|.
  \end{align*}
  Thus, by Lemma~\ref{l:prop1}(\ref{e:l1a}), we obtain
    \begin{align}
    \label{e:l2-5}
    d(f,q')+1< \frac{\lambda}{1/2-\lambda}d(p',q')=
    \frac{2\lambda}{1-2\lambda}d(p',q').
  \end{align}
  \medskip

 Finally, combining the two cases (``Up" and ``Down") above, that is, combining (\ref{e:l2-1}), (\ref{e:l2-2}) and (\ref{e:l2-5}), we have that every edge $f\in E(N_e)\setminus \ag$ is contained in the neighborhood of radius
 \begin{align*}
   \frac{2\lambda}{1-2\lambda}\cdot d(p',q')
 \end{align*}
 around the set $\{ p',q'\}$ of endpoints of $N_e$.
 Thus,  we obtain
  \begin{align}
  \label{e:loc1}
  |E(N_e)\cap \ag|\geqslant d(p',q') - 2\cdot\frac{2\lambda}{1-2\lambda}d(p',q')=
  \frac{1-6\lambda}{1-2\lambda}|E(N_e)|.
  \end{align}

\medskip

The formula above is perfectly satisfactory in the case $\lambda < 1/6$. However for $\lambda=1/6$ we need to provide a more precise bound, studying in more details the case ``Down".
  \medskip

  \noindent
  \emph{Case ``Down+":} As in the case ``Down" we have that $r_f\neq r$. Again, we may assume that $q''$ lies (on $\gamma$) between $f$ and $f'$ --- see Figure~\ref{fig_3}.
  Moreover, we consider now only one of the vertices $p',q'$, assuming that $q'$ lies (on $\gamma$) between $e$ and $e'$, as in Lemma~\ref{l:prop1} --- see Figures~\ref{fig_5} \& \ref{fig_6}.
  Let $s$ be furthest from $q'$ vertex in $r'\cap \gamma \setminus r$.
By considerations from the case ``Down" we have that $q''$ lies between $q'$ and $q$. We consider separately two subcases.
\medskip

\noindent
\emph{Subcase 1:} $q''$ lies between $q'$ and $s$ --- see Figure~\ref{fig_5}.
\begin{figure}[h!]
\centering
\includegraphics[width=0.8\textwidth]{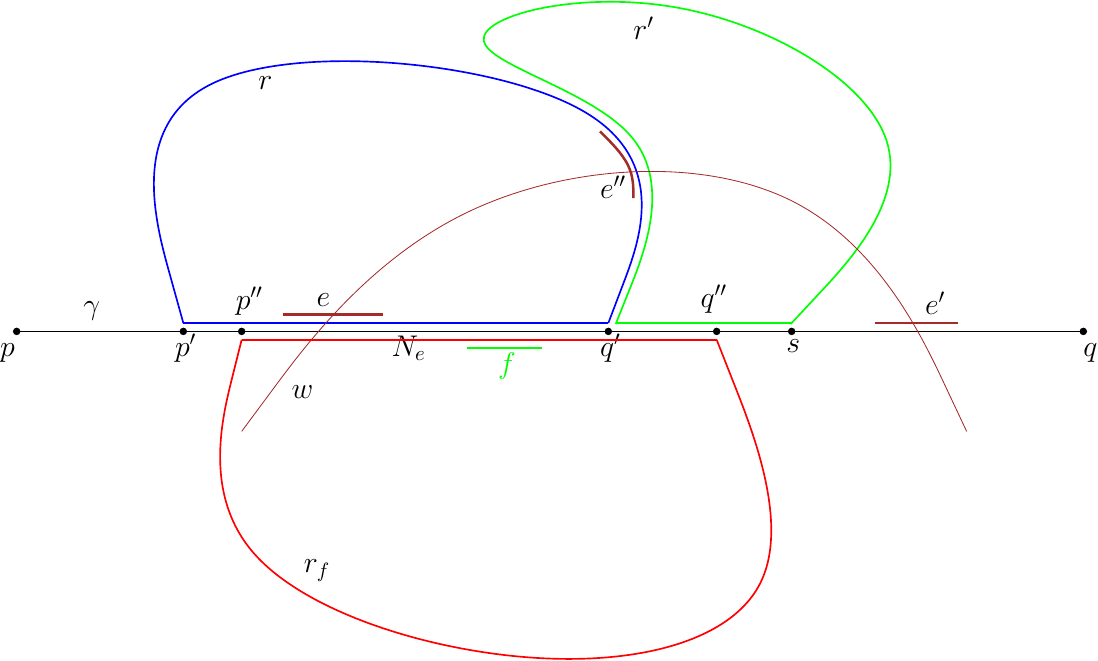}
\caption{Subcase 1 of Case ``Down+".}
\label{fig_5}
\end{figure}
In this case the path between $f$ and $q'$, including $f$, and the path between $q'$ and $q''$ are pieces, so that, by the $C'(\lambda)$--condition we have:
\begin{align*}
  1+d(f,q'')=1+d(f,q')+d(q',q'')< 2 \lambda |r_f|.
\end{align*}

However, by Lemma~\ref{l:prop1}(\ref{e:l1a}) we have
$d(f,q'')>(1/2-\lambda)|r_f|$.  This leads to contradiction for $\lambda \leqslant 1/6$.

\medskip

\noindent
\emph{Subcase 2:} $s$ lies between $q'$ and $q''$ --- see Figure~\ref{fig_6}.
\begin{figure}[h!]
\centering
\includegraphics[width=0.8\textwidth]{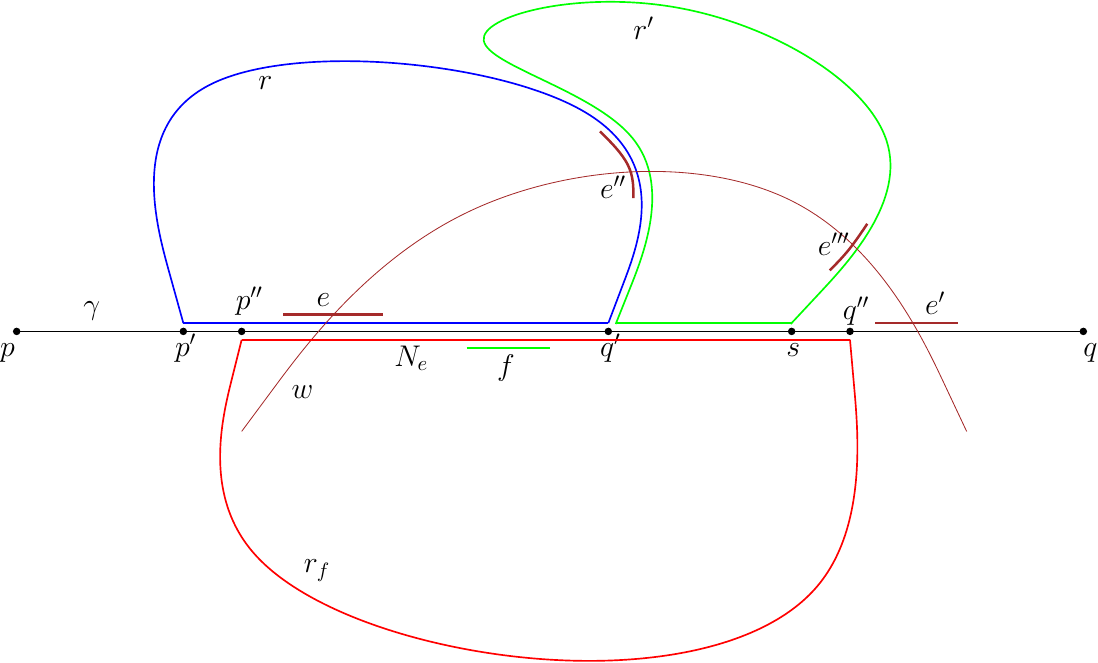}
\caption{Subcase 2 of Case ``Down+".}
\label{fig_6}
\end{figure}
Let $e'''$ be the  vertex in $\Gamma_w$ adjacent to $e''$ and on the geodesic (in $\Gamma_w$) between $e$ and $e'$. Observe that we may have $e'''=e'$ or $e'''\neq e'$, however both cases can be treated at once.
The path in $r \cap r'$ between $e''$ and $q'$, including $e''$, is a piece. Similarly, the path in $r'$ between $e'''$ and $s$, including $e'''$ and omitting $e''$, is a piece.
Since also the path between $q'$ and $s$ is a piece, by the $C'(\lambda)$--condition we have
\begin{align*}
  \frac{|r'|}{2}+1=1+d(e'',q')+d(q',s)+d(s,e''')+1
  <3\lambda |r'|,
\end{align*}
that leads to contradiction for $\lambda \leqslant 1/6$.
\medskip

Combining the two subcases we have that there are no edges like $f$ in the neighborhood of one of points $p'$, $q'$.
\medskip

Now we combine all the cases: ``Up'', ``Down" and ``Down+", i.e.\ the formulas: (\ref{e:l2-1}), (\ref{e:l2-2}), (\ref{e:l2-5}).
We conclude that any edge $f\in E(N_e)\setminus \ag$ is contained in the $[2\lambda d(p',q')]$--neighborhood around one of vertices $p',q'$ or in the $\{[(2\lambda)/(1-2\lambda)]\cdot d(p',q')\}$--neighbor\-hood around the other vertex.
Thus, similarly as in (\ref{e:loc1}), we obtain
\begin{align*}
  |E(N_e)\cap \ag|\geqslant d(p',q') - \frac{2\lambda}{1-2\lambda}d(p',q')- 2\lambda d(p',q')=
  \frac{1-6\lambda+4\lambda ^2}{1-2\lambda}|E(N_e)|.
  \end{align*}

\end{proof}

\subsection{Linear separation property}
\label{s:global}
In this subsection we estimate the overall density of edges with walls separating $p$ and $q$, thus obtaining the linear separation property.
We use the local estimate on the density of $\ag$ (see Lemma~\ref{l:prop2}) and the local-to-global density principle (Lemma~\ref{l:lgd}).

\begin{tw}[Linear separation property]
  \label{p:lsp}
  For any two vertices $p,q$ in $X$ we have
  \begin{align*}
    d(p,q)\geqslant \dw(p,q)\geqslant \frac{1-6\lambda+4\lambda^2}{2-4\lambda}\cdot
    d(p,q),
  \end{align*}
  i.e., the path metric and the wall pseudo-metric are bi-Lipschitz equivalent.
\end{tw}

\begin{proof}
  The left inequality is clear. Now we prove the right one.
  Let $\gamma$ be a geodesic joining $p$ and $q$.
  The number $|E(\gamma)|$ of edges in $\gamma$ is equal to $d(p,q)$. On the other hand, the number $|\ag|$ of edges in $\gamma$ whose walls meet $\gamma$ in only one edge is at most $\dw(p,q)$. We will thus bound $|\ag|$ from below.

  For any edge $e$ of $\gamma$, let $N_e$ be its relator neighborhood. The collection $\mathcal U=\{N_e\;|\; e\in E(\gamma)\}$ forms a covering family of subpaths of $\gamma$.
  By the local estimate (Lemma~\ref{l:prop2}) we have that
  \begin{align*}
    \frac{|\ag\cap E(N_e)|}{|E(N_e)|}\geqslant \frac{1-6\lambda+4\lambda^2}{1-2\lambda}.
  \end{align*}
  Thus, by the local-to-global density principle (Lemma~\ref{l:lgd}), we have
  \begin{align*}
    |\ag|\geqslant \frac{1}{2}\cdot \frac{1-6\lambda+4\lambda^2}{1-2\lambda}\cdot|E(\gamma)|,
  \end{align*}
  that finishes the proof.
\end{proof}

\begin{rem}
A detailed  description of the geometry of infinitely presented groups satisfying the stronger small cancelation condition $C'(1/8)$
is provided in a recent work of Dru\c{t}u and the first author \cite{ADr}. This yields many analytic and geometric properties of such groups.
In particular, an alternative proof of the bi-Lipschitz equivalence between 
the wall pseudo-metric and the word length metric is given
for such groups. This uses the standard decomposition of the group elements developed in that paper (a powerful technical tool of independent interest).
\end{rem}

\section{Haagerup property}
\label{s:Haa}

A consequence of the linear separation property (Theorem~\ref{p:lsp}) is the following.

\begin{tw}
  \label{t:haag}
  Let $G$ be a group acting properly on a simply connected $C'(1/6)$--complex $X$. Then $G$ acts properly on a space with 
walls. In particular, $G$ has the Haagerup property.
\end{tw}
\begin{proof}
The group $G$ acts properly on the set of vertices $\xz$ of $X$ equipped with the path metric $d(\cdot,\cdot)$.
By Proposition~\ref{p:sww}, this action gives rise to the action by automorphisms on the space with walls $(\xz,\mathcal W)$.
By the linear separation property (Theorem~\ref{p:lsp}), for $\lambda \leqslant 1/6$, we conclude that $G$ acts properly on $(\xz,\mathcal W)$.
By an observation of Bo\. zejko-Januszkiewicz-Spatzier \cite{BoJaS} and Haglund-Paulin-Valette (cf. \cite{ChMV}), the group $G$ has the Haagerup property.
\end{proof}

Observe that Main Theorem follows immediately from the above, since the group $G$ given by the presentation (\ref{e:p1}) acts properly on its Cayley complex $X$, as described in Section~\ref{s:prelim}.

Since infinite groups with the Haagerup property do not satisfy Kazhdan's property~(T), we obtain the following strengthening of \cite[Theorem 14.2]{W-sc} (which was actually proved under weaker
$B(6)$--condition) in the $C'(1/6)$--condition case.

\begin{cor}
  \label{c:noT}
  Let an infinite group $G$ act properly on a simply connected $C'(1/6)$--complex. Then $G$ does not have Kazhdan's property (T).
\end{cor}

Another application of the linear separation property concerns orbits of group actions on classical normed spaces. 
If $G$ acts properly on a space with walls $(\xz,\mathcal W)$, then it has a proper affine isometric action 
on the space $L^p$, for every $1\leqslant p <\infty$,  of $p$-summable functions on the family of half-spaces determined by walls $\mathcal W$, cf.~\cite[Proposition 3.1]{CTV} and~\cite[Corollary 1.5]{CDH}.
If $b\colon G\to L^p$ denotes the 1-cocycle of such an action, then $\|b(g)\|_p=\dw(gx_0,x_0)^{1/p}$, where  $x_0\in\xz$ is a base point.
The set $b(G)$ is the orbit of $0\in L^p$. 
Therefore, Theorem~\ref{p:lsp} immediately implies the following.
\begin{cor}
  \label{c:orbit}
  Let an infinite group $G$ act properly on a simply connected $C'(1/6)$--complex. Then $G$ acts properly by affine isometries on the space $L^p$, 
  for every $1\leqslant p <\infty$, with orbit growth bi-Lipschitzly equivalent to $|g|^{1/p}$, where $|g|=d(gx_0, x_0)$.
\end{cor}
In particular, $G$ acts properly by affine isometries on the space $L^1$ with bi-Lipschitzly embedded orbits.
If $G$ is non-amenable, then the above orbit growth in $L^2$ is the best possible. Indeed, by a result of Guentner-Kaminker \cite[Theorem 5.3]{GuKa},
the orbit growth strictly larger than $|g|^{1/2}$ implies the amenability of the group.

\section{Final remarks --- relations to work of D.\ Wise}
\label{s:final}

The main tool used in this paper is the system of walls for a simply connected complex satisfying the $C'(1/6)$--condition, introduced by D.\ Wise in \cite{W-sc}, and then developed further e.g.\ in \cite{W-qch}. In fact, Wise uses this tool usually to treat more general complexes --- complexes satisfying the  $B(6)$--condition.
One might be tempted to claim that the results provided in this paper follow immediately from Wise's work. We show here that this is not the case. Nevertheless, we follow of course many of the ideas presented in \cite{W-sc,W-qch}.
\medskip

First, although many results in \cite{W-sc} concern the general case of $B(6)$--complexes (compare e.g.\ Section~\ref{s:walls} above), eventually some finiteness conditions appear when dealing with proper group actions. For example, in
\cite[Theorem 14.1 and Theorem 14.2]{W-sc} non-satisfiability of the Kazhdan's property (T) is proved under additional assumptions about cocompactness or freeness of the action. Our Corollary~\ref{c:noT} does not require such assumptions. Under our assumptions (stronger than $B(6)$--condition)  --- i.e.\ with the $C'(1/6)$--condition --- we may use the linear separation property (Theorem~\ref{p:lsp}) to omit additional restrictions. However, the linear separation property does not hold for all simply connected $B(6)$--complexes.
Moreover, for such complexes there is, in general, no lower bound on the wall pseudo-metric in terms of the
path metric, as the following example shows.
\medskip

\noindent
{\bf Example 1.}
\begin{figure}[h!]
\centering
\includegraphics[width=0.77\textwidth]{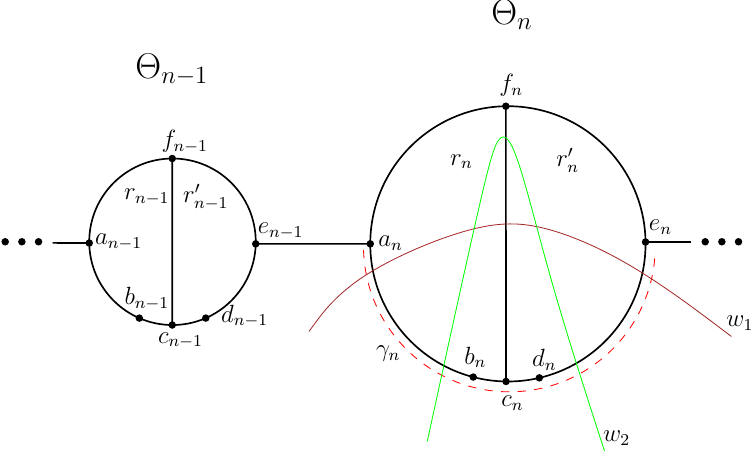}
\caption{Example 1.}
\label{fig_7}
\end{figure}
Let $X^{(1)}$ be constructed using an infinite union of graphs $\Theta_n$, for $n=1,2,3,\ldots$, depicted in Figure~\ref{fig_7}. There is an edge joining $e_n$ with $a_{n+1}$, for every $n$ making $\xj$ connected. For each $\Theta_n$, there are two $2$--cells: $r_n,r_n'$ attached to the shortest simple cycles in $\Theta$, respectively: to the cycle passing through $a_n,b_n,c_n,f_n$, and through $c_n,d_n,e_n,f_n$.
The obtained $2$--complex $X$ is simply connected.
We assign the lengths (in the path metric) of segments in $\Theta_n$ as follows:
\begin{align*}
  d(b_n,c_n)=d(c_n,d_n)=3,\\
  d(c_n,f_n)=2n,\\
  d(a_n,b_n)=d(d_n,e_n)=n,\\
  d(a_n,f_n)=d(f_n,e_n)=n+3.
\end{align*}
It is easy to check that this turns $X$
into a $B(6)$--complex. Now, consider the standard structure of the space with walls $(\xz,\mathcal W)$, as defined in Section~\ref{s:walls}.
The only walls separating $a_n$ from $e_n$ are the walls containing the edges in the segments $b_nc_n$ and $c_nd_n$, each of length $3$ --- see Figure~\ref{fig_7} (two other edges $w_1,w_2$ double intersecting the geodesic $\gamma$ between $a_n$ and $e_n$, thus not separating them are depicted). Hence we obtain $\dw(a_n,e_n)=6$, while $d(a_n,e_n)=2n+6 \rightarrow \infty$, as $n\rightarrow \infty$.
\medskip

We do not know whether a group acting properly on a $B(6)$--complex acts properly on the corresponding space with walls.
\medskip

On the other hand the linear separation property is proved in \cite[Theorem 5.45]{W-qch} for complexes satisfying a condition being some strengthening of the $B(6)$--condition (in the context of a more general small cancellation theory). The proof goes roughly as follows. For a geodesic $\gamma$ and for its edge $e_1$, whose wall does not separate endpoints of $\gamma$ (compare our proof in Section~\ref{s:properness}) ``there is (an edge) $e_2$ within a uniformly distance of $e_1$" whose wall separates the endpoints of $\gamma$. This works clearly only in the case of finitely many types of $2$--cells as the following example shows.

\medskip

\noindent
{\bf Example 2.} Let $X$ be a complex consisting of two $2$--cells $r,r'$, meeting along a (piece) segment $a,q'$.
\begin{figure}[h!]
\centering
\includegraphics[width=0.7\textwidth]{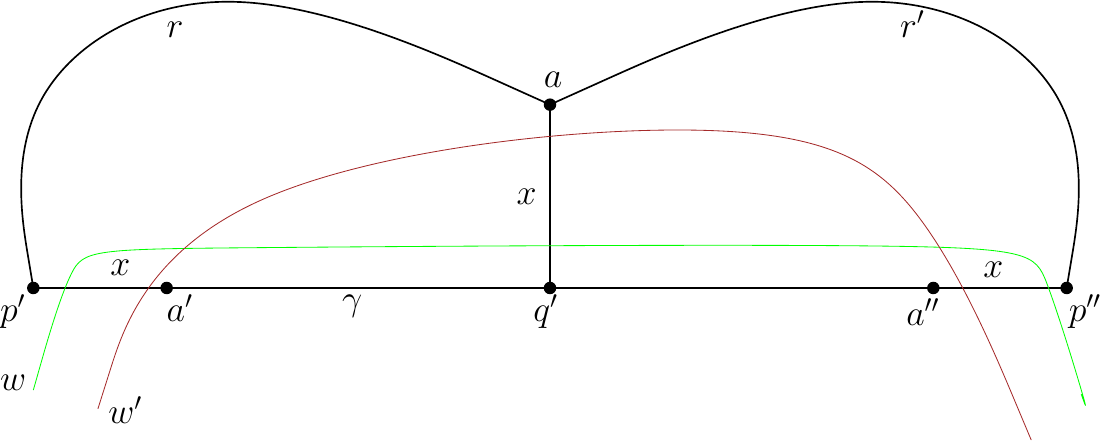}
\caption{Example 2.}
\label{fig_4}
\end{figure}
We set the following lengths (in the path metric) on $X$:
\begin{align*}
  x=d(a,q')=2d(a',p')=2d(a'',p''),\\
  d(q',a')=d(q',a'')=\frac{|r|}{2}-x=\frac{|r'|}{2}-x.
\end{align*}
Making the ratio $x/|r|$ small we can turn $X$ into a $C'(\lambda)$--complex for arbitrarily small $\lambda>0$. On the other hand, all the (standard) walls containing edges in the segment $p'a'$ do not separate $p'$ from $p''$, double crossing the geodesic $\gamma$ between those two points (two such walls $w,w'$ are depicted in Figure~\ref{fig_7}). Thus, with $x$ growing (which can happen if there is infinitely many types of $2$--cells in a complex), for an edge in $p'a'$ its big neighborhood may consist of edges whose walls do not separate $\gamma$.


\end{document}